\documentclass{kms-j}

\newcommand{\publname}{J.~Korean Math.~Soc.}
\issueinfo{}
  {}
  {}
  {}
\pagespan{1}{}
\copyrightinfo{}
  {Korean Mathematical Society}

\usepackage{graphicx}
\allowdisplaybreaks

\theoremstyle{plain}
\newtheorem{theorem}{Theorem}[section]

\newtheorem{lemma}[theorem]{Lemma}

\theoremstyle{definition}
\newtheorem*{definition}{Definition}

\theoremstyle{remark}
\newtheorem{remark}[theorem]{Remark}

\begin{document}

\title[Radii problems for the generalized Mittag-Leffler functions]
{Radii problems for the generalized Mittag-Leffler functions}

\author[A. Prajapati]{Anuja Prajapati}
\address{Anuja Prajapati \\ Department of  Mathematics \\ Sambalpur University \\Jyoti Vihar, Burla, 768019, Sambalpur, Odisha,  India}
\email{anujaprajapati49@gmail.com}
\thanks{This work was supported by INSPIRE fellowship, Department of Science and
Technology, New Delhi, Government of India}


\begin{abstract}
In this paper our aim is to find  various  radii problems  of the generalized Mittag-Leffler function for three different kinds of normalization by using their Hadamard factorization in such a way that the resulting functions are analytic. The basic tool of this study is the Mittag-Leffler function in series. Also we have shown that the obtained radii are the smallest positive roots of some functional equations. 
\end{abstract}

\maketitle

\section{Introduction}

 Geometric function theory and special functions are close related to each other, since hypergeometric functions have been used in the proof of the famous Bieberbach conjecture. Due to this conjecture various authors have considered  some  geometric properties of special functions such as Bessel, Lommel, Struve, $q$-Bessel functions, which can be expressed by the hypergeometric series. The first important results on the geometric properties of Mittag-Leffler function and other special functions can be found in \cite{brown,todd,merkes,wilf,bansal}. Actually there are  relationship between the geometric properties and the zeros of the special functions. Numerous authors has been done their works on the zeros of the special functions mentioned earlierly. It is well  known that the concepts of convexity, starlikeness, close-to-convexity and uniform convexity including necessary and sufficient conditions have a long history as a part of geometric function theory. Recently, radius problems with some geometric properties like univalence, starlikeness, convexity, uniform convexity, parabolic starlikeness, close-to-convexity, strongly starlikeness of Wright, Bessel, Struve, Lommel functions of the first kind have been investigated in \cite{baricz4,baricz5,baricz1,baricz2,baricz,baricz6,ponnusamy,baricz7,yagmur,deniz,prajapati,deniz1,gang,aktas,bohra}. Recently, the radii of convexity and starlikness of the generalized  Mittag-Leffler functions were studied by Baricz and Prajapati  \cite{prajapati}. Motivated by the above results and using the technique of Baricz et al.\cite{baricz} in this paper, our aim is to find some new results for the various radii problems  of $\eta-$uniformly convexity, $\eta-$parabolic starlikeness, $\alpha-$convexity and strong starlikeness of order $\rho$ for the three different kinds of normalization  of the generalized  Mittag-Leffler function.  
 
\subsection{Characterization of uniform convex and parabolic starlike functions} In order to present our results we need the following definitions. Let $\mathbb{D}(r)$ be the open  disk $\{z \in \mathbb{C}:|z|<r\},$ where $r>0,$ and set $\mathbb{D}=\mathbb{D}(1).$ Let $(a_{n}),\quad n \geq 2$ be a sequence of complex numbers with
$$d=\limsup_{n \rightarrow \infty}|a_{n}|^\frac{1}{n} \geq 0,  \quad\text{and}\quad r_{f}=\frac{1}{d},$$
where $r_{f}$ means the radius of convergence of the series $f(z).$ If $d=0$  then $r_{f}=+\infty.$  Moreover, let $\mathcal{A}$ be the class of analytic functions $f:\mathbb{D}(r_{f})\rightarrow \mathbb{C},$ of the form
\begin{equation}\label{uncv}
f(z)=z+\sum_{n=2}^{\infty}a_{n}z^{n}.
\end{equation}
Let $\mathcal{S}$ be the class of functions which belongs to $\mathcal{A}$ that are univalent in $\mathbb{D}(r).$ The class of convex functions, denoted by $\mathcal{C},$ is the subclass of $\mathcal{S}$ which consists of functions $f$ for which the image domain $f(\mathbb{D}(r))$ is convex domain. The real numbers
 $$r^{c}(f)=\sup \left\{r>0:\Re \left\{1+\frac{zf^{\prime\prime}(z)}{f^{\prime}(z)}\right\}>0~~ \text{for}~~~ \text{all}~~ z \in \mathbb{D}(r)\right\}$$
 and
 $$r^{c}_{\rho}(f)=\sup \left\{r>0:\Re \left\{1+\frac{zf^{\prime\prime}(z)}{f^{\prime}(z)}\right\}>\rho~~ \text{for}~~~ \text{all}~~ z \in \mathbb{D}(r)\right\}$$ are called the radius of convexity and the radius of convexity of order $\rho$ of the function $f,$ respectively. We note that $r^{c}(f)=r^{c}_{0}(f)$ is the largest radius such that the image region  $f(\mathbb{D}(r^{c}(f)))$ is a convex domain. For more details about convex functions refer to Duren book \cite{duren} and to the references therein. A function is said to be uniformly convex in $\mathbb{D}$ if $f(z)$ is in class of convex functions and has the property that for every circular arc $\varepsilon$ contained in $\mathbb{D},$ with center $\kappa$ also in $\mathbb{D},$ the arc $f(\varepsilon)$ is a convex arc.
In 1993, Ronning \cite{ronning} determined necessary and sufficient conditions of analytic functions to be uniformly convex in the open unit disk, while in 2002, Ravichandran \cite{ravi} also presented simpler criterion for uniform convexity. 
\begin{definition}
 Let $f(z)$ be the form (\ref{uncv}). Then $f$ is a uniformly convex functions if and only if 
 \begin{equation*}
 \Re \left(1+\frac{zf^{\prime\prime}(z)}{f^{\prime}(z)}\right)>\left|\frac{zf^{\prime\prime}(z)}{f^{\prime}(z)}\right|,\quad z \in \mathbb{D}.
\end{equation*}
The concept of the radius of uniform convexity is defined in \cite{deniz}. 
\begin{equation*}
r^{ucv}(f)=\sup \left\{r \in (0,r_{f}):\Re\left(1+\frac{zf^{\prime\prime}(z)}{f^{\prime}(z)}\right)>\left|\frac{zf^{\prime\prime}(z)}{f^{\prime}(z)}\right|,z \in \mathbb{D}(r)\right\}.
\end{equation*} 
\end{definition} 
 A function is said to be in the class of $\eta-$uniformly convex function of order $\rho,$ denoted by $\eta-UCV(\rho)$ in \cite{bharti} if
 \begin{equation}\label{ucv12}
\Re\left\{1+\frac{zf^{\prime\prime}(z)}{f^{\prime}(z)}\right\} >\eta \left|\frac{zf^{\prime\prime}(z)}{f^{\prime}(z)}\right|+\rho,\quad \eta \geq 0, \rho \in [0,1), z \in \mathbb{D}. 
 \end{equation}
 These classes generalize various other classes. The class $\eta-UCV(0)=\eta-UCV$ is the class of $\eta-$uniformly convex functions \cite{kanas} also see (\cite{kanas1,kanas2}) and $1-UCV(0)=UCV$ is the class of uniformly convex functions defined by Goodman \cite{goodman} and ronning \cite{ronning}, respectively. The radius of $\eta-$uniform convexity of order $\rho$ is defined  in \cite{deniz1}.
$$r^{\eta-ucv(\rho)}_{f}=\sup \left\{r \in (0,r_{f}):\Re \left\{1+\frac{zf^{\prime\prime}(z)}{f^{\prime}(z)}\right\}>\eta \left|\frac{zf^{\prime\prime}(z)}{f^{\prime}(z)}\right|+\rho; \eta \geq 0, \rho \in [0,1),z \in \mathbb{D}(r)\right\}.$$
 \begin{definition}
 Let $f(z)$ be in the form (\ref{uncv}). Then we say that $f$ is parabolic starlike function if and only if 
 $$\Re \left(\frac{zf^{\prime}(z)}{f(z)}\right)>\left|\frac{zf^{\prime}(z)}{f(z)}-1\right|, \quad
  (z \in \mathbb{D}).$$
In 1993, Ronning \cite{ronning} introduced  the class of parabolic starlike and it is denoted by $SP.$ The class $SP$ is a subclass of the class of starlike functions of order 1/2 and the class of strongly starlike functions of order 1/2.\\
A function is said to be the class of $\eta-$parabolic starlike function of order $\rho,$ denoted by $\eta-SP(\rho)$ in \cite{kanas} if 
\begin{equation}\label{sp}
 \Re \left(\frac{zf^{\prime}(z)}{f(z)}\right)> \eta\left|\frac{zf^{\prime}(z)}{f(z)}-1\right|+\rho, \quad \eta \geq 0, \rho \in [0,1),
  z \in \mathbb{D}.
\end{equation}
These classes generalizes to other classes. The class $\eta-SP(0)=\eta-SP$ is the class of $\eta-$parabolic starlike functions \cite{kanas} and $1-SP(0)=SP$ is the class of parabolic starlike function.
The radius of $\eta-$parabolic starlikeness of order $\rho$ is defined  in \cite{bohra}. 
$$r^{\eta-SP(\rho)}_{f}=\sup \left\{r \in (0,r_{f}):\Re \left(\frac{zf^{\prime}(z)}{f(z)}\right)> \eta\left|\frac{zf^{\prime}(z)}{f(z)}-1\right|+\rho;~\eta \geq 0, \rho \in [0,1),
  z \in \mathbb{D}(r)\right\}.$$
 \end{definition}
 
 \subsection{Geometric interpretation}
$f\in \eta-UCV(\rho)$ and $f \in \eta-SP(\rho)$ if and only if $1+\frac{zf^{\prime\prime}(z)}{f^{\prime}(z)}$ and $\frac{zf^{\prime}(z)}{f(z)}$ respectively take all the values in the conic domain $\mathcal{R}_{\eta,\rho}$ which is included in the right half plane such that
\begin{equation}\label{k444}
\mathcal{R}_{\eta,\rho}=\{w=u+iv:u>\eta \sqrt{(u-1)^{2}+v^{2}}+\rho\}.
\end{equation}
 Denote $\mathcal{P}(P_{\eta,\rho})~~(\eta\geq 0,0 \leq \rho<1),$ the family of functions $p,$ such that $p\in \mathcal{P}$ and $p\prec P_{\eta,\rho}$ in $\mathbb{D},$ where $\mathcal{P}$ denotes the class of Caratheodory functions and the functions $P_{\eta,\rho}$ maps the unit disk conformally onto the domain $\mathcal{R}_{\eta,\rho}$ such that $1 \in \mathcal{R}_{\eta,\rho}$ and $\partial \mathcal{R}_{\eta,\rho}$ is a curve defined by the equality$$\partial \mathcal{R}_{\eta,\rho}=\{u+iv:u^{2}=(\eta \sqrt{(u-1)^{2}+v^{2}}+\rho)^{2}\}.$$
From the elementary computations we see that $\partial \mathcal{R}_{\eta,\rho}$ represents the conic sections symmetric about real axis.
\begin{itemize}
\item{} $\mathcal{R}_{\eta,\rho}$  curve reduces to the imaginary axis for $\eta=0.$
\item{}$\mathcal{R}_{\eta,\rho}$ is an elliptic region for $\eta>1.$
\item{} $\mathcal{R}_{\eta,\rho}$ is a parabolic domain for $\eta=1.$
\item{}$\mathcal{R}_{\eta,\rho}$ is a hyperbolic domain for  $0<\eta<1.$
\end{itemize}

\begin{definition}
Let $f(z)$ be the form (\ref{uncv}). For $\alpha \in \mathbb{R}$ and $\rho \in [0,1).$ Then we say that $f$ is $\alpha-$convex of order $\rho$ in $\mathbb{D}(r)$ if and only if
$$\Re \left((1-\alpha)\frac{zf^{\prime}(z)}{f(z)}+\alpha \left(1+\frac{zf^{\prime\prime}(z)}{f^{\prime}(z)}\right)\right)>\rho, \quad z \in \mathbb{D}.$$
\end{definition}
The radius of $\alpha-$convexity of order $\rho$ of the function $f$ is defined by the real number,
$$r_{\alpha,\rho}(f)=\sup\left\{r \in (0,r_{f}):\Re \left((1-\alpha)\frac{zf^{\prime}(z)}{f(z)}+\alpha \left(1+\frac{zf^{\prime\prime}(z)}{f^{\prime}(z)}\right)\right)>\rho; \rho \in [0,1), z \in \mathbb{D}(r)\right\}.$$
 The concept of $\alpha-$convexity is defined in \cite{mocanu1}.  Radius of $\alpha-$convexity of order $\rho$ is the generalization of the radius of starlikeness of order $\rho$ and of the radius of convexity of order $\rho.$ We have $r_{0,\rho}(f)=r^{*}_{\rho}(f)$ and $r_{1,\rho}(f)=r^{c}_{\rho}(f).$ For the detailed treatment on starlike, convex and $\alpha-$convex functions we refer to \cite{mocanu,duren, baricz6,caglar,mocanu1} and to the references therein.

 The above definitions are used to determine the radii of $\eta-$uniformly convexity, $\alpha-$convexity and $\eta-$parabolic starlikeness of order $\rho$ for the functions of the form (\ref{uncv}). Also we will need the following lemma in the sequel.
 \begin{lemma}\cite{deniz}\label{ucv1}
 If $a>b>r\geq |z|,$ and $\lambda \in [0,1],$ then 
 \begin{equation}\label{uc1}
 \left|\frac{z}{b-z}-\lambda \frac{z}{a-z}\right| \leq \frac{r}{b-r}-\lambda \frac{r}{a-r}.
 \end{equation}
The followings  are very simple consequences of the inequality
 \begin{equation}\label{uc2}
\Re  \left(\frac{z}{b-z}-\lambda \frac{z}{a-z}\right) \leq \frac{r}{b-r}-\lambda \frac{r}{a-r}
 \end{equation}
 and 
 \begin{equation}\label{uc3}
\Re  \left(\frac{z}{b-z}\right)\leq \left|\frac{z}{b-z}\right| \leq \frac{r}{b-r}.
 \end{equation}
 \end{lemma}

\subsection{The three parameter generalization of Mittag-Leffler function}
 Consider the function $\phi(\omega,z)$ defined by, 
 \begin{equation}\label{radd1} 
\phi(\omega,z)=\sum_{k\geq 0}\frac{z^{n}}{\Gamma(\omega k+1)},\quad \omega>0, z \in \mathbb{D},
 \end{equation}
 where $\Gamma$ denotes the Gamma function.  In 1903, Mittag-Leffler \cite{mittag} was introduced this function and thereafter it is  known as the Mittag-Leffler function. Depending upon the two complex parameters $\omega$ and $\beta,$ in 1905 Wiman \cite{wiman}  introduced the another version of Mittag-Leffler function which having similar properties with the function $\phi(\omega,z).$ It is defined by the following series,
 \begin{equation}\label{radd2}
 \phi(\omega,\beta,z)=\sum_{k\geq 0} \frac{z^{k}}{\Gamma(\omega k+\beta)},\quad \omega,\beta>0,z \in \mathbb{D}.
 \end{equation}

 In 1971, Prabhakar \cite{prabhakar} introduced the three parameter function $\phi(\omega,\beta,\gamma,z)$ in the form of
 \begin{equation}\label{radd3}
 \phi(\omega,\beta.\gamma,z)=\sum_{k \geq 0}\frac{(\gamma)_{k}z^{k}}{k!\Gamma(\omega k+\beta)},\quad \omega,\beta,\gamma>0, z \in \mathbb{D},
 \end{equation}
 where $(\gamma)_{k}$ denotes the Pochhammer symbol (or shifted factorial) given in terms of the Gamma function by $(d)_{k}=\Gamma (d+k)/\Gamma(d).$
Some particular cases of $\phi(\omega,\beta,\gamma,z)$ are given in \cite{prajapati}.
Observe that the function $z \rightarrow \phi(\omega,\beta,\gamma,-z^{2})$ does not belong to $\mathcal{A}.$ Thus first we perform some natural normalization. We define three functions originating from $\phi(\omega,\beta,\gamma,z):$

$$f_{\omega,\beta,\gamma}(z)=\left(z^{\beta}\Gamma(\beta)\phi(\omega,\beta,\gamma,-z^{2})\right)^{1/\beta},$$
$$g_{\omega,\beta,\gamma}(z)=z\Gamma(\beta)\phi(\omega,\beta,\gamma,-z^{2}),$$
$$h_{\omega,\beta,\gamma}(z)=z\Gamma(\beta)\phi(\omega,\beta,\gamma,-z).$$
Obviously these functions belong to the class $\mathcal{A}.$ Of course, there exist infinitely many other normalization. The main motivation to consider the above ones is the studied normalization in the literature of Bessel, Struve, Lommel and Wright functions.
\subsection{ Preliminary result on the Mittag-Leffler function}
First we define, three transformations mapping the set
$\left\{\left(\frac{1}{\omega},\beta\right):\omega>1,\beta>0\right\}$ into itself:
$$A:\left(\frac{1}{\omega},\beta\right)\to \left(\frac{1}{2\omega},\beta\right), \ \ \ B:\left(\frac{1}{\omega},\beta\right)\to \left(\frac{1}{2\omega},\omega+\beta\right),$$
$$C:\left(\frac{1}{\omega},\beta\right)\to \left\{\begin{array}{ll}\displaystyle\left(\frac{1}{\omega},\beta-1\right),&\mbox{if}\ \ \beta>1\\ \\ \displaystyle\left(\frac{1}{\omega},\beta\right),&\mbox{if}\ \ 0<\beta\leq1\end{array}\right..$$
We put $W_b=A(W_a)\cup B(W_a),$ where
$$W_a=\left\{\left(\frac{1}{\omega},\beta\right):1<\omega<2,\beta\in[\omega-1,1]\cup[\omega,2]\right\},$$
and $W_i$ is denoted as the smallest set containing $W_b$ and invariant with respect to $A,$ $B$ and $C,$ that is, if $(a,b)\in W_i,$ then
$A(a,b),$ $B(a,b),$ $C(a,b)\in W_i.$ By using a result of Peresyolkova \cite{peres}, Kumar and Pathan \cite{kumar} recently proved that if
$\left(\frac{1}{\omega},\beta\right)\in W_i$ and $\gamma>0,$ then all zeros of the generalized Mittag-Leffler function $\phi(\omega,\beta,\gamma,z)$ are real and negative. It is worth mentioning that
the reality of the zeros as well as their distribution in the case of $\gamma=1,$ that is of Wiman's extension $\phi(\omega,\beta,z),$ has a rich literature. For more details see  the papers of Dzhrbashyan \cite{dzhr}, Ostrovski\u{i} and Peresyolkova \cite{ostrov}, Popov and Sedletskii \cite{sedlet}.
\section{Main Results}
\subsection{Radii of $\eta-$ uniformly convexity of order $\rho$ of functions $f_{\omega,\beta,\gamma},g_{\omega,\beta,\gamma}$ and $h_{\omega,\beta,\gamma}$ }
Now, our aim is to investigate the radii of $\eta-$uniformly convexity of order $\rho$ of the normalized forms of the generalized three parameter Mittag-Leffler function, that is of  $f_{\omega,\beta,\gamma},g_{\omega,\beta,\gamma}$ and $h_{\omega,\beta,\gamma}.$ The technique of determining the radii of $\eta-$uniformly convexity of order $\rho$ in the next theorem follows the ideas from \cite{deniz}.
\begin{theorem}
Let $(\frac{1}{\omega},\beta) \in W_{i},~~\gamma>0$ and $\rho \in [0,1).$
\begin{itemize}
\item[\bf{a}.] The radius of $\eta-$uniform convexity of order $\rho$ of the function $f_{\omega,\beta,\gamma}$  is the smallest positive root of the equation
 $$1-\rho+(1+\eta)\left(\frac{r\Psi^{\prime\prime}_{\omega,\beta,\gamma}(r)}{\Psi^{\prime}_{\omega,\beta,\gamma}(r)}+\left(\frac{1}{\beta}-1\right)\frac{r\Psi^{\prime}_{\omega,\beta,\gamma}(r)}{\Psi_{\omega,\beta,\gamma}(r)}\right)=0,$$
 where $\Psi_{\omega,\beta,\gamma}(z)=z^{\beta}\lambda_{\omega,\beta,\gamma}(z).$
 \item[\bf{b}.] The radius of $\eta-$uniform convexity of order $\rho$ of the function $g_{\omega,\beta,\gamma}$ is the smallest  positive root of the equation
 $$1-\rho+(1+\eta)\frac{rg^{\prime\prime}_{\omega,\beta,\gamma}(r)}{g^{\prime}_{\omega,\beta,\gamma}(r)}=0.$$
 \item[\bf{c}.] The radius of $\eta-$uniform convexity of order $\rho$ of $h_{\omega,\beta,\gamma}$ is the smallest  positive root of the equation
 $$1-\rho+(1+\eta)\frac{rh^{\prime\prime}_{\omega,\beta,\gamma}(r)}{h^{\prime}_{\omega,\beta,\gamma}(r)}=0.$$
\end{itemize}
\end{theorem}
\begin{proof}
\begin{itemize}
\item[\bf{a}.] Let $\varsigma_{\omega,\beta,\gamma,n}$ and $\varsigma^{\prime}_{\omega,\beta,\gamma,n}$ be the $n$th positive roots of $\Psi_{\omega,\beta,\gamma}$ and $\Psi^{\prime}_{\omega,\beta,\gamma}$ respectively. In \cite[Theorem-3(a)]{prajapati}, the following equality was demonstrated,
$$ 1+\frac{zf^{\prime\prime}_{\omega,\beta,\gamma}(z)}{f^{\prime}_{\omega,\beta,\gamma}(z)}=1-\left(\frac{1}{\beta}-1\right)\sum_{n \geq 1}\frac{2z^{2}}{\varsigma^{2}_{\omega,\beta,\gamma,n}-z^{2}}-\sum_{n \geq 1}\frac{2z^{2}}{\varsigma^{\prime 2}_{\omega,\beta,\gamma,n}-z^{2}}.$$
In order to prove this theorem we need to investigate two different cases such as $\beta \in (0,1]$ and $\beta>1.$ First suppose $\beta \in (0,1].$ In this case, with the help of (\ref{uc3}) for $\beta \in (0,1]$, we have
\begin{align}\label{rad17}
\Re \left(1+\frac{zf^{\prime\prime}_{\omega,\beta,\gamma}(z)}{f^{\prime}_{\omega,\beta,\gamma}(z)}\right)&\geq 1-\left(\frac{1}{\beta}-1\right)\sum_{n \geq 1}\frac{2r^{2}}{\varsigma^{2}_{\omega,\beta,\gamma,n}-r^{2}}-\sum_{n \geq 1}\frac{2r^{2}}{\varsigma^{\prime 2}_{\omega,\beta,\gamma,n}-r^{2}}\\ \nonumber &=1+\frac{rf^{\prime\prime}_{\omega,\beta,\gamma}(r)}{f^{\prime}_{\omega,\beta,\gamma}(r)},
\end{align}
 where $|z|\leq r<\varsigma^{\prime}_{\omega,\beta,\gamma,1}<\varsigma_{\omega,\beta,\gamma,1},$ holds true for $|z|=r.$  Moreover, in view of (\ref{uc3}) and $\eta \geq 0$, we get,
 \begin{eqnarray}\label{uc4}
\eta\left|\frac{zf^{\prime\prime}_{\omega,\beta,\gamma}(z)}{f^{\prime}_{\omega,\beta,\gamma}(z)}\right| &=&\eta\left|\left(\frac{1}{\beta}-1\right)\sum_{n \geq 1}\frac{2z^{2}}{\varsigma^{2}_{\omega,\beta,\gamma,n}-z^{2}}-\sum_{n \geq 1}\frac{2z^{2}}{\varsigma^{\prime 2}_{\omega,\beta,\gamma,n}-z^{2}}\right|\\ 
\nonumber &\leq& \eta\sum_{n \geq 1}\left(\left(\frac{1}{\beta}-1\right)\frac{2r^{2}}{\varsigma^{2}_{\omega,\beta,\gamma,n}-r^{2}}-\frac{2r^{2}}{\varsigma^{\prime 2}_{\omega,\beta,\gamma,n}-r^{2}}\right)\\ \nonumber &=&-\eta\frac{rf^{\prime\prime}_{\omega,\beta,\gamma}(r)}{f^{\prime}_{\omega,\beta,\gamma}(r)},
 \end{eqnarray}
where $|z|\leq r< \varsigma^{\prime}_{\omega,\beta,\gamma,1}<\varsigma_{\omega,\beta,\gamma,1}.$ In view of the inequality (\ref{uc2}), we obtain that  (\ref{rad17}) and (\ref{uc4}) are also valid when $\beta \geq 1$   for all $z \in (0,\varsigma^{\prime}_{\omega,\beta,\gamma,1}).$ Here we used that the zeros of $\varsigma_{\omega,\beta,\gamma,1}$ and $\varsigma^{\prime}_{\omega,\beta,\gamma,1}$ interlace, that is, we have $\varsigma^{\prime}_{\omega,\beta,\gamma,1}<\varsigma_{\omega,\beta,\gamma,1}.$ 
 From (\ref{rad17}) and (\ref{uc4}), we have 
 \begin{equation}\label{uc5}
\Re \left(1+\frac{zf^{\prime\prime}_{\omega,\beta,\gamma}(z)}{f^{\prime}_{\omega,\beta,\gamma}(z)}\right)-\eta \left|\frac{zf^{\prime\prime}_{\omega,\beta,\gamma}(z)}{f^{\prime}_{\omega,\beta,\gamma}(z)}\right|-\rho\geq  1-\rho+(1+\eta) \frac{rf^{\prime\prime}_{\omega,\beta,\gamma}(z)}{f^{\prime}_{\omega,\beta,\gamma}(z)},  
 \end{equation}
where $|z|\leq r < \varsigma^{\prime}_{\omega,\beta,\gamma}$ and $\rho \in [0,1),$ $\eta \geq 0.$
 Due to minimum principle for harmonic functions, equality holds if and only if $z=r.$ Now, the above deduced inequality imply for $r \in (0, \varsigma^{\prime}_{\omega,\beta,\gamma,1}).$
 $$\inf_{|z|<r}\left\{\Re \left(1+\frac{zf^{\prime\prime}_{\omega,\beta,\gamma}(z)}{f^{\prime}_{\omega,\beta,\gamma}(z)}\right)-\eta\left|\frac{zf^{\prime\prime}_{\omega,\beta,\gamma}(z)}{f^{\prime}_{\omega,\beta,\gamma}(z)}\right|-\rho\right\}=1-\rho+(1+\eta)\frac{rf^{\prime\prime}_{\omega,\beta,\gamma}(r)}{f^{\prime}_{\omega,\beta,\gamma}(r)}.$$
 The function $u_{\omega,\beta,\gamma}:(0,\varsigma^{\prime}_{\omega,\beta,\gamma,1}) \rightarrow \mathbb{R},$ defined by 
$$u_{\omega,\beta,\gamma}(r)=1-\rho+(1+\eta)\frac{rf^{\prime\prime}_{\omega,\beta,\gamma}(r)}{f^{\prime}_{\omega,\beta,\gamma}(r)}=1-\rho-(1+\eta)\sum_{n\geq 1}\left(\frac{2r^{2}}{\varsigma^{\prime 2}_{\omega,\beta,\gamma,n}-r^{2}}-\left(1-\frac{1}{\beta}\right)\frac{2r^{2}}{\varsigma^{2}_{\omega,\beta,\gamma,n}-r^{2}}\right)$$
 is strictly decreasing when $\beta \in (0,1]$ and $\eta \geq 0,$ $\rho \in [0,1).$ Moreover, it is also strictly decreasing when $\beta >1$ since
 \begin{eqnarray*}
 u^{\prime}_{\omega,\beta,\gamma}(r)&=&-(1+\eta)\left(\left(\frac{1}{\beta}-1\right)\sum_{n\geq 1}\frac{8r\varsigma^{2}_{\omega,\beta,\gamma,n}}{(\varsigma^{2}_{\omega,\beta,\gamma,n}-r^{2})^{2}}-\sum_{n\geq 1}\frac{8r\varsigma^{\prime 2}_{\omega,\beta,\gamma,n}}{(\varsigma^{\prime 2}_{\omega,\beta,\gamma,n}-r^{2})^{2}}\right)\\ &<&(1+\eta)\sum_{n\geq 1}\left(\frac{8r\varsigma^{2}_{\omega,\beta,\gamma,n}}{(\varsigma^{2}_{\omega,\beta,\gamma,n}-r^{2})^{2}}-\frac{8r\varsigma^{\prime 2}_{\omega,\beta,\gamma,n}}{(\varsigma^{\prime 2}_{\omega,\beta,\gamma,n}-r^{2})^{2}}\right)<0,
 \end{eqnarray*}
 for $r \in (0, \varsigma^{\prime}_{\omega,\beta,\gamma,1}).$ 
  
 Observe that $\lim_{r\searrow 0} u_{\omega,\beta,\gamma}(r)=1-\rho>0$ and $\lim_{r\nearrow \varsigma^{\prime}_{\omega,\beta,\gamma}} u_{\omega,\beta,\gamma}(r)=-\infty.$ Thus it follows that the equation  
  
 $$1+(1+\eta)\frac{rf^{\prime\prime}_{\omega,\beta,\gamma}(r)}{f^{\prime}_{\omega,\beta,\gamma}(r)}=\rho, \quad \eta \geq 0, \rho \in [0,1)$$ has a unique root  situated in $r_{1}\in (0,\varsigma^{\prime}_{\omega,\beta,\gamma,1}).$
 \item[\bf{b}.]
 Let $\upsilon_{\omega,\beta,\gamma,n}$ be the $n$th positive zero of the function  $g^{\prime}_{\omega,\beta,\gamma}.$ In \cite[Theorem-3 (b)]{prajapati}, the following equality was proved:
\begin{equation}\label{uc6} 
 1+\frac{zg^{\prime\prime}_{\omega,\beta,\gamma}(z)}{g^{\prime}_{\omega,\beta,\gamma}(z)}=1-\sum_{n \geq 1}\frac{2z^{2}}{\upsilon^{2}_{\omega,\beta,\gamma,n}-z^{2}},
 \end{equation}
 and it was shown in \cite{prajapati}.
 \begin{equation}\label{uc7}
 \Re \left(1+\frac{zg^{\prime\prime}_{\omega,\beta,\gamma}(z)}{g^{\prime}_{\omega,\beta,\gamma}(z)}\right)\geq 1-\sum_{n \geq 1}\frac{2r^{2}}{\upsilon^{2}_{\omega,\beta,\gamma,n}-r^{2}},\quad
  |z|\leq r< \upsilon_{\omega,\beta,\gamma,1}.
 \end{equation}
From equality (\ref{uc6}) and $\eta \geq 0,$ we have,
 \begin{align}\label{uc8}
\eta \left|\frac{zg^{\prime\prime}_{\omega,\beta,\gamma}(z)}{g^{\prime}_{\omega,\beta,\gamma}(z)}\right|= \eta\left|\sum_{n \geq 1}\frac{2z^{2}}{\upsilon^{2}_{\omega,\beta,\gamma,n}-z^{2}}\right|= -\eta\frac{rg^{\prime\prime}_{\omega,\beta,\gamma}(r)}{g^{\prime}_{\omega,\beta,\gamma}(r)},~~|z|\leq r < \upsilon_{\omega,\beta,\gamma,1}.
\end{align}
 By using the inequalities  (\ref{uc7}) and (\ref{uc8}) we obtain 
\begin{equation*}
\Re \left(1+\frac{zg^{\prime\prime}_{\omega,\beta,\gamma}(z)}{g^{\prime}_{\omega,\beta,\gamma}(z)}\right)-\eta\left|\frac{zg^{\prime\prime}_{\omega,\beta,\gamma}(z)}{g^{\prime}_{\omega,\beta,\gamma}(z)}\right|-\rho\geq 1-\rho+(1+\eta)\frac{rg^{\prime\prime}_{\omega,\beta,\gamma}(r)}{g^{\prime}_{\omega,\beta,\gamma}(r)},
\end{equation*}
$|z|\leq r < \upsilon_{\omega,\beta,\gamma,1},\eta \geq 0,\rho \in [0,1)$.
According to minimum principle for harmonic functions, equality holds if and only if  $z=r.$ Thus, for $ r \in (0,\upsilon_{\omega,\beta,\gamma,1}),$ $\eta \geq 0$ and $\rho \in [0,1)$  we get
 $$\inf_{|z|<r}\left\{\Re \left(1+\frac{zg^{\prime\prime}_{\omega,\beta,\gamma}(z)}{g^{\prime}_{\omega,\beta,\gamma}(z)}\right)-\eta\left|\frac{zg^{\prime\prime}_{\omega,\beta,\gamma}(z)}{g^{\prime}_{\omega,\beta,\gamma}(z)}\right|-\rho\right\}=1-\rho+(1+\eta)\frac{rg^{\prime\prime}_{\omega,\beta,\gamma}(r)}{g^{\prime}_{\omega,\beta,\gamma}(r)}.$$
The function $u_{\omega,\beta,\gamma
}:(0,\upsilon_{\omega,\beta,\gamma,1})\rightarrow \mathbb{R},$ defined by
$$u_{\omega,\beta,\gamma}(r)=1-\rho+(1+\eta)\frac{rg^{\prime\prime}_{\omega,\beta,\gamma}(r)}{g^{\prime}_{\omega,\beta,\gamma}(r)},$$
is strictly decreasing and  
$\lim_{r\searrow 0} u_{\omega,\beta,\gamma}(r)=1-\rho>0$ and $\lim_{r\nearrow \upsilon_{\omega,\beta,\gamma}} u_{\omega,\beta,\gamma}(r)=-\infty.$
Consequently, the equation $1+(1+\eta)\frac{rg^{\prime\prime}_{\omega,\beta,\gamma}(r)}{g^{\prime}_{\omega,\beta,\gamma}(r)}=\rho,$ $\eta \geq 0$ and $\rho \in [0,1)$ has a unique root $r_{2}$ in $(0,\upsilon_{\omega,\beta,\gamma,1}).$
\item[\bf{c}.]
Let $\upsilon_{\omega,\beta,\gamma,n}$ denote the $n$th positive zero of the function  $h_{\omega,\beta,\gamma}(z)$. In \cite[Theorem-3(c)]{prajapati}, the following equations was  obtained
\begin{equation}\label{uc9}
\frac{zh^{\prime\prime}_{\omega,\beta,\gamma}(z)}{h^{\prime}_{\omega,\beta,\gamma}(z)}=-\sum_{n \geq 1}\frac{z}{\upsilon_{\omega,\beta,\gamma,n}-z},
\end{equation}
and in the same paper with the help of (\ref{uc9}) the following inequality was given 
\begin{equation}\label{uc10}
\Re\left(1+\frac{zh^{\prime\prime}_{\omega,\beta,\gamma}(z)}{h^{\prime}_{\omega,\beta,\gamma}(z)}\right)\geq 1+\frac{rh^{\prime\prime}_{\omega,\beta,\gamma}(r)}{h^{\prime}_{\omega,\beta,\gamma}(r)},\quad |z|<r<\upsilon_{\omega,\beta,\gamma,1}.
\end{equation}
From (\ref{uc9}), we have
\begin{align}\label{uc11}
\eta\left|\frac{zh^{\prime\prime}_{\omega,\beta,\gamma}(z)}{h^{\prime}_{\omega,\beta,\gamma}(z)}\right|=\eta\left|\sum_{n \geq 1}\frac{z}{\upsilon_{\omega,\beta,\gamma,n}-z}\right|=-\eta\frac{rh^{\prime\prime}_{\omega,\beta,\gamma}(r)}{h^{\prime}_{\omega,\beta,\gamma}(r)},\quad \eta \geq 0, |z|<r<\upsilon_{\omega,\beta,\gamma,1}.
\end{align}
From (\ref{uc10}) and (\ref{uc11}), we have
\begin{equation*}
\Re\left(1+\frac{zh^{\prime\prime}_{\omega,\beta,\gamma}(z)}{h^{\prime}_{\omega,\beta,\gamma}(z)}\right)-\eta\left|\frac{zh^{\prime\prime}_{\omega,\beta,\gamma}(z)}{h^{\prime}_{\omega,\beta,\gamma}(z)}\right|-\rho\geq 1-\rho+(1+\eta)\frac{rh^{\prime\prime}_{\omega,\beta,\gamma}(r)}{h^{\prime}_{\omega,\beta,\gamma}(r)},
\end{equation*}
$\quad |z|<r<\upsilon_{\omega,\beta,\gamma,1},\eta \geq 0, \rho \in [0,1).$
Due to the minimum principle for harmonic function and equality holds if and only if $z=r.$ Thus,  we have
$$\inf_{|z|<r}\left\{\Re \left(1+\frac{zh^{\prime\prime}_{\omega,\beta,\gamma}(z)}{h^{\prime}_{\omega,\beta,\gamma}(z)}\right)-\eta\left|\frac{zh^{\prime\prime}_{\omega,\beta,\gamma}(z)}{h^{\prime}_{\omega,\beta,\gamma}(z)}\right|-\rho\right\}=1-\rho+(1+\eta)\frac{rh^{\prime\prime}_{\omega,\beta,\gamma}(r)}{h^{\prime}_{\omega,\beta,\gamma}(r)}.$$ 
For every  $r \in (0, \upsilon_{\omega,\beta,\gamma,1}),$ $\eta \geq 0,$ and $\rho \in [0,1).$ Since the function $w_{\omega,\beta,\gamma}(r):(0, \upsilon_{\omega,\beta,\gamma,1})\rightarrow \mathbb{R}$ defined by
\end{itemize}
$$w_{\omega,\beta,\gamma}(r)=1-\rho-(1+\eta)\sum_{n\geq 1} \frac{r}{\upsilon_{\omega,\beta,\gamma,n}-r},$$
as decreasing on $(0,\upsilon_{\omega,\beta,\gamma,1})$
and  
$\lim_{r \searrow 0}w_{\omega,\beta,\gamma}(r)=1-\rho,\quad \lim_{r \nearrow \upsilon_{\omega,\beta,\gamma}}w_{\omega,\beta,\gamma}(r)=-\infty.$
It follows that the equation $w_{\omega,\beta,\gamma}(r)=0$ has a unique root $r_{3} \in (0, \upsilon_{\omega,\beta,\gamma,1})$ and this root is the radius of uniform convexity.
\end{proof}
\subsection{Radii of $\alpha-$convexity of order $\rho$ for the functions $f_{\omega,\beta,\gamma},g_{\omega,\beta,\gamma}$ and $h_{\omega,\beta,\gamma}$}
Now, we are going to investigate the radii of $\alpha-$convexity of order $\rho$ of the functions $f_{\omega,\beta,\gamma},g_{\omega,\beta,\gamma}$ and $h_{\omega,\beta,\gamma}.$ The technique used in the process of finding the radii of $\alpha-$convexity of order $\rho$ in the next theorem is based on the ideas from \cite{mocanu} and \cite{baricz6}. The results of the theorem are natural extensions of some recent results see \cite{prajapati}, where the special case of $\alpha=1$ and $\rho=0$ was considered  on the Mittag-Leffler function. For proving the main results we will use the following notation,
$$J(\alpha,u(z))=(1-\alpha)\frac{zu^{\prime}(z)}{u(z)}+\alpha \left(1+\frac{zu^{\prime\prime}(z)}{u^{\prime}(z)}\right).$$
\begin{theorem}\label{theorem1}
If $(\frac{1}{\omega},\beta) \in W_{i},\gamma>0,  \alpha \geq 0,\rho
\in [0,1),$
 then the radius of $\alpha-$convexity of order $\rho$ of the function $f_{\omega,\beta,\gamma}$  is the smallest positive root of the equation
 $$\alpha+\alpha\frac{r\Psi^{\prime\prime}_{\omega,\beta,\gamma}(r)}{\Psi^{\prime}_{\omega,\beta,\gamma}(r)}+\left(\frac{1}{\beta}-\alpha\right)\frac{r\Psi^{\prime}_{\omega,\beta,\gamma}(r)}{\Psi_{\omega,\beta,\gamma}(r)}=\rho,$$
 where $\Psi_{\omega,\beta,\gamma}(z)=z^{\beta}\lambda_{\omega,\beta,\gamma}(z).$  The radius of $\alpha-$convexity satisfies $r_{\alpha,\rho}(f_{\omega,\beta,\gamma})<\varsigma^{\prime}_{\omega,\beta,\gamma,1}<\varsigma_{\omega,\beta,\gamma,1},$ where $\varsigma_{\omega,\beta,\gamma,1}$ and $\varsigma^{\prime}_{\omega,\beta,\gamma,1}$ denote the first positive zeros of $\Psi_{\omega,\beta,\gamma}$ and $\Psi^{\prime}_{\omega,\beta,\gamma}$ respectively. Moreover the function $\alpha \rightarrow r_{\alpha,\rho}(f_{\omega,\beta,\gamma})$ is strictly decreasing on $[0,\infty)$ and  consequently, we have $r^{c}_{\rho}<r_{\alpha,\rho}(f_{\omega,\beta,\gamma})<r^{*}_{\rho}(f_{\omega,\beta,\gamma})$ for all $\alpha \in (0,1),\rho \in [0,1)$.
 \end{theorem}
 \begin{proof}
 Without loss of generality we assume that $\alpha>0,$ the case $\alpha=0$ was proved in \cite{prajapati}. By using the definition of the function $f_{\omega,\beta,\gamma}(z)$ we have
 $$\frac{zf^{\prime}_{\omega,\beta,\gamma}(z)}{f_{\omega,\beta,\gamma}(z)}=\frac{1}{\beta}\frac{z\Psi^{\prime}_{\omega,\beta,\gamma}(z)}{\Psi_{\omega,\beta,\gamma}(z)},1+\frac{zf^{\prime\prime}_{\omega,\beta,\gamma}(z)}{f^{\prime}_{\omega,\beta,\gamma}(z)}=1+\frac{z\Psi^{\prime\prime}_{\omega,\beta,\gamma}(z)}{\Psi^{\prime}_{\omega,\beta,\gamma}(z)}+\left(\frac{1}{\beta}-1\right)\frac{z\Psi^{\prime}_{\omega,\beta,\gamma}(z)}{\Psi_{\omega,\beta,\gamma}(z)}.$$
 Now consider the following infinite product representations in  \cite{prajapati},
  \begin{gather*}
 \Gamma(\beta)\Psi_{\omega,\beta,\gamma}(z)=z^{\beta}\prod_{n \geq 1}\left(1-\frac{z^{2}}{\varsigma^{2}_{\omega,\beta,\gamma,n}}\right),\quad \Gamma(\beta)\Psi^{\prime}_{\omega,\beta,\gamma}(z)=\beta z^{\beta-1}\prod_{n \geq 1}\left(1-\frac{z^{2}}{\varsigma^{\prime2}_{\omega,\beta,\gamma,n}}\right),
 \end{gather*}
where $\varsigma_{\omega,\beta,\gamma,1}$ and $\varsigma^{\prime}_{\omega,\beta,\gamma,1}$ denote the first positive zeros of $\Psi_{\omega,\beta,\gamma}$ and $\Psi^{\prime}_{\omega,\beta,\gamma}$ respectively. By logarithmic differentiation we have
 \begin{gather*}
 \frac{z\Psi^{\prime}_{\omega,\beta,\gamma}(z)}{\Psi_{\omega,\beta,\gamma}(z)}=\beta-\sum_{n \geq 1}\frac{2z^{2}}{\varsigma^{2}_{\omega,\beta,\gamma,n}-z^{2}},\quad \frac{z\Psi^{\prime\prime}_{\omega,\beta,\gamma}(z)}{\Psi^{\prime}_{\omega,\beta,\gamma}(z)}=\beta-1-\sum_{n \geq 1}\frac{2z^{2}}{\varsigma^{\prime 2}_{\omega,\beta,\gamma,n}-z^{2}},
 \end{gather*}
which imply that
\begin{eqnarray*}
J(\alpha,f_{\omega,\beta,\gamma}(z))&=&(1-\alpha)\frac{zf^{\prime}_{\omega,\beta,\gamma}(z)}{f_{\omega,\beta,\gamma}(z)}+\alpha \left(1+\frac{zf^{\prime\prime}_{\omega,\beta,\gamma}(z)}{f^{\prime}_{\omega,\beta,\gamma}(z)}\right)\\ \nonumber &=& 1-\left(\frac{1}{\beta}-\alpha \right)\sum_{n \geq 1}\frac{2z^{2}}{\varsigma^{2}_{\omega,\beta,\gamma,n}-z^{2}}-\alpha \sum_{n \geq 1}\frac{2z^{2}}{\varsigma^{\prime 2}_{\omega,\beta,\gamma,n}-z^{2}}.
\end{eqnarray*}
We know that if $a>b>0, z \in \mathbb{C}$ and $\lambda \leq 1 ,$ then for all $|z|<b$ we have, 
\begin{equation}\label{ucv21}
\lambda \Re \left(\frac{z}{a-z}\right)-\Re \left(\frac{z}{b-z}\right)\geq \frac{|z|}{a-|z|}-\frac{|z|}{b-|z|}.
\end{equation}
From Lemma \ref{ucv1}, it is clear that it is assumed that $\lambda \in [0,1],$ we do not need the assumption $\lambda \geq 0$ so using (\ref{ucv21}) for all $z \in \mathbb{D}(0,\varsigma^{\prime}_{\omega,\beta,\gamma,1}).$ We obtain the inequality
\begin{eqnarray*}
\frac{1}{\alpha}\Re(J(\alpha,f_{\omega,\beta,\gamma}(z)))&\geq & \frac{1}{\alpha}+\left(1-\frac{1}{\alpha\beta}\right)\sum_{n \geq 1}\frac{2r^{2}}{\varsigma^{2}_{\omega,\beta,\gamma,n}-r^{2}}- \sum_{n \geq 1}\frac{2r^{2}}{\varsigma^{\prime 2}_{\omega,\beta,\gamma,n}-r^{2}}\\\nonumber &=& \frac{1}{\alpha}(J(\alpha,f_{\omega,\beta,\gamma}(r))),
\end{eqnarray*}
where $|z|=r.$  The zeros  $\varsigma_{\omega,\beta,\gamma}$ and $\varsigma^{\prime}_{\omega,\beta,\gamma}$ are interlacing. From \cite[Lemma-1]{prajapati}, we have  
\begin{equation}\label{ucv32}
\varsigma_{\omega,\beta,\gamma}<\varsigma^{\prime}_{\omega,\beta,\gamma}.
\end{equation}
The above inequality implies that for $r \in (0,\varsigma^{\prime}_{\omega,\beta,\gamma,1}),$ we have $\inf_{z \in \mathbb{D}(r)}J(\alpha,f_{\omega
,\beta,\gamma}(z))=J(\alpha,f_{\omega,\beta,\gamma}(r))$ and the function $r \rightarrow J(\alpha,f_{\omega,\beta,\gamma}(r))$ is strictly decreasing on $(0,\varsigma^{\prime}_{\omega,\beta,\gamma,1}).$ Since
\begin{eqnarray*}
\frac{\partial }{\partial  r}J(\alpha,f_{\omega,\beta,\gamma}(r))&=&-\left(\frac{1}{\beta}-\alpha \right)\sum_{n\geq 1}\frac{4r\varsigma^{2}_{\omega,\beta,\gamma,n}}{(\varsigma^{2}_{\omega,\beta,\gamma,n}-r^{2})^{2}}-\alpha\sum_{n\geq 1}\frac{4r\varsigma^{\prime 2}_{\omega,\beta,\gamma,n}}{(\varsigma^{\prime 2}_{\omega,\beta,\gamma,n}-r^{2})^{2}}\\ &<&\alpha\sum_{n\geq 1}\frac{4r\varsigma^{2}_{\omega,\beta,\gamma,n}}{(\varsigma^{2}_{\omega,\beta,\gamma,n}-r^{2})^{2}}-\alpha\sum_{n\geq 1}\frac{4r\varsigma^{\prime 2}_{\omega,\beta,\gamma,n}}{(\varsigma^{\prime 2}_{\omega,\beta,\gamma,n}-r^{2})^{2}}<0,
\end{eqnarray*}
for $\gamma>0$ and $r \in (0,\varsigma^{\prime}_{\omega,\beta,\gamma,1}).$ Again we  used that the zeros $\varsigma_{\omega,\beta,\gamma,n}$ and $\varsigma^{\prime}_{\omega,\beta,\gamma,n}$ are interlaced and for  all $n \in \mathbb{N},\gamma>0,$ and $r<\sqrt{\varsigma_{\omega,\beta,\gamma,1}. ~\varsigma^{\prime}_{\omega,\beta,\gamma,1}}$ we have that,
$$\varsigma^{2}_{\omega,\beta,\gamma,n}(\varsigma^{\prime 2}_{\omega,\beta,\gamma,n}-r^{2})^{2}<\varsigma^{\prime 2}_{\omega,\beta,\gamma,n}(\varsigma^{2}_{\omega,\beta,\gamma,n}-r^{2})^{2}.$$
We also have that $ \lim_{r\searrow 0}J(\alpha,f_{\omega,\beta,\gamma}(r))=1>\rho $ and $\lim_{r\nearrow \varsigma^{\prime}_{\omega,\beta,\gamma,1}}J(\alpha,f_{\omega,\beta,\gamma}(r))=-\infty,$
which means that for $z \in \mathbb{D}(r_{1}).$ We have $\Re J(\alpha,f_{\omega,\beta
,\gamma}(r))>\rho$ if and only if $r_{1}$ is the unique root of $J(\alpha,f_{\omega,\beta
,\gamma}(r))=\rho,$ situated in $(0,\varsigma^{\prime}_{\omega,\beta,\gamma,1}).$ Finally using again the interlacing inequalities (\ref{ucv32}), we obtain the inequality
$$\frac{\partial}{\partial \alpha} J(\alpha,f_{\omega,\beta,\gamma}(r))=\sum_{n \geq 1}\frac{2r^{2}}{\varsigma^{2}_{\omega,\beta,\gamma,n}-r^{2}}-\sum_{n \geq 1}\frac{2r^{2}}{\varsigma^{\prime 2}_{\omega,\beta,\gamma,n}-r^{2}}<0,$$
where $\gamma>0, \alpha \geq 0$ and $r \in (0, \varsigma^{\prime}_{\omega,\beta,\gamma,1}).$ This implies that the function $\alpha \rightarrow J(\alpha,f_{\omega,\beta,\gamma}(r))$ is strictly decreasing on $[0,\infty)$ for all $\gamma>0, \alpha \geq 0$ and $r \in (0, \varsigma^{\prime}_{\omega,\beta,\gamma,1})$ fixed. Consequently, as a function of $\alpha$ the unique root of the equation $J(\alpha,f_{\omega,\beta,\gamma}(r))=\rho$ is strictly decreasing where $\rho \in [0,1), \gamma>0$ and $r \in (0, \varsigma^{\prime}_{\omega,\beta,\gamma,1}).$ Thus in the case when $\alpha \in (0,1)$ the radius of $\alpha-$convexity of the function $f_{\omega,\beta,\gamma}$ will be between the radius of convexity and the radius of starlikeness of the function $f_{\omega,\beta,\gamma}.$ 
 \end{proof}
\begin{theorem}\label{theorem2}
If $(\frac{1}{\omega},\beta) \in W_{i},\gamma>0,  \alpha \geq 0,\rho
\in [0,1),$
 then the radius of $\alpha-$convexity of order $\rho$ of the function $g_{\omega,\beta,\gamma}$  is the smallest positive root of the equation
 $$(1-\alpha)\frac{rg^{\prime}_{\omega,\beta,\gamma}(r)}{g_{\omega,\beta,\gamma}(r)}+\alpha\left(1+\frac{r g^{\prime\prime}_{\omega,\beta,\gamma}(r)}{g^{\prime}_{\omega,\beta,\gamma}(r)}\right)=\rho.$$  The radius of $\alpha-$convexity satisfies $r_{\alpha,\rho}(g_{\omega,\beta,\gamma})<\upsilon^{\prime}_{\omega,\beta,\gamma,1}<\upsilon_{\omega,\beta,\gamma,1},$ where $\upsilon_{\omega,\beta,\gamma,1}$ and $\upsilon^{\prime}_{\omega,\beta,\gamma,1}$ denote the first positive zeros of $g_{\omega,\beta,\gamma}$ and $g^{\prime}_{\omega,\beta,\gamma}$ respectively. Moreover the function $\alpha \rightarrow r_{\alpha,\rho}(g_{\omega,\beta,\gamma})$ is strictly decreasing on $[0,\infty)$ and  consequently, we have $r^{c}_{\rho}<r_{\alpha,\rho}(g_{\omega,\beta,\gamma})<r^{*}_{\rho}(g_{\omega,\beta,\gamma})$ for all $\alpha \in (0,1),\rho \in [0,1)$.

\end{theorem}
\begin{proof}
Similarly, as the proof of Theorem \ref{theorem1}, we assume that $\alpha>0,$ the case $\alpha=0$ was proved in \cite{prajapati}. By using the definition of the function $g_{\omega,\beta,\gamma}(z)$ and the infinite product representation, we have 
 $$1+\frac{zg^{\prime\prime}_{\omega,\beta,\gamma}(z)}{g^{\prime}_{\omega,\beta,\gamma}(z)}=1-\sum_{n \geq 1}\frac{2z^{2}}{\upsilon^{2}_{\omega,\beta,\gamma,n}-z^{2}},$$
where $\upsilon_{\omega,\beta,\gamma,n}$ is the nth positive zero of the $g_{\omega,\beta,\gamma}.$ Thus, we have
\begin{eqnarray*}
J(\alpha,g_{\omega,\beta,\gamma}(z))&=&(1-\alpha)\frac{zg^{\prime}_{\omega,\beta,\gamma}(z)}{g_{\omega,\beta,\gamma}(z)}+\alpha \left(1+\frac{zg^{\prime\prime}_{\omega,\beta,\gamma}(z)}{g^{\prime}_{\omega,\beta,\gamma}(z)}\right)\\ \nonumber &=& 1 -\left(1-\alpha \right)\sum_{n \geq 1}\frac{2z^{2}}{\upsilon^{2}_{\omega,\beta,\gamma,n}-z^{2}}-\alpha \sum_{n \geq 1}\frac{2z^{2}}{\upsilon^{\prime 2}_{\omega,\beta,\gamma,n}-z^{2}}.
\end{eqnarray*}
Applying the inequality (\ref{ucv21}), we have
\begin{eqnarray*}
\frac{1}{\alpha}\Re(J(\alpha,g_{\omega,\beta,\gamma}(z)))&\geq & \frac{1}{\alpha}+\left(1-\frac{1}{\alpha}\right)\sum_{n \geq 1}\frac{2r^{2}}{\upsilon^{2}_{\omega,\beta,\gamma,n}-r^{2}}- \sum_{n \geq 1}\frac{2r^{2}}{\upsilon^{\prime 2}_{\omega,\beta,\gamma,n}-r^{2}}\\\nonumber &=& \frac{1}{\alpha}(J(\alpha,g_{\omega,\beta,\gamma}(r))),
\end{eqnarray*}
where $|z|=r.$
The above inequality implies that for $r \in (0,\upsilon^{\prime}_{\omega,\beta,\gamma,1}),$ we have $\inf_{z \in \mathbb{D}(r)}J(\alpha,g_{\omega
,\beta,\gamma}(z))=J(\alpha,g_{\omega,\beta,\gamma}(r))$ and the function $r \rightarrow J(\alpha,g_{\omega,\beta,\gamma}(r))$ is strictly decreasing on $(0,\upsilon^{\prime}_{\omega,\beta,\gamma,1}).$ Since
\begin{eqnarray*}
\frac{\partial }{\partial  r}J(\alpha,g_{\omega,\beta,\gamma}(r))&=&-\left(1-\alpha \right)\sum_{n\geq 1}\frac{4r\upsilon^{2}_{\omega,\beta,\gamma,n}}{(\upsilon^{2}_{\omega,\beta,\gamma,n}-r^{2})^{2}}-\alpha\sum_{n\geq 1}\frac{4r\upsilon^{\prime 2}_{\omega,\beta,\gamma,n}}{(\upsilon^{\prime 2}_{\omega,\beta,\gamma,n}-r^{2})^{2}}<0,
\end{eqnarray*}
for $\gamma>0$ and $r \in (0,\upsilon^{\prime}_{\omega,\beta,\gamma,1}).$ Again we  used that the zeros $\upsilon_{\omega,\beta,\gamma,n}$ and $\upsilon^{\prime}_{\omega,\beta,\gamma,n}$ are interlaced and for  all $n \in \mathbb{N},\gamma>0,$ and $r<\sqrt{\upsilon_{\omega,\beta,\gamma,1}.~\upsilon^{\prime}_{\omega,\beta,\gamma,1}}$ we have that,
$$\upsilon^{2}_{\omega,\beta,\gamma,n}(\upsilon^{\prime 2}_{\omega,\beta,\gamma,n}-r^{2})^{2}<\upsilon^{\prime 2}_{\omega,\beta,\gamma,n}(\upsilon^{2}_{\omega,\beta,\gamma,n}-r^{2})^{2}.$$
We also have that $ \lim_{r\searrow 0}J(\alpha,g_{\omega,\beta,\gamma}(r))=1>\rho $ and $\lim_{r\nearrow \upsilon^{\prime}_{\omega,\beta,\gamma,1}}J(\alpha,g_{\omega,\beta,\gamma}(r))=-\infty,$
which means that for $z \in \mathbb{D}(r_{2}).$ We have $\Re J(\alpha,g_{\omega,\beta
,\gamma}(r))>\rho$ if and only if $r_{1}$ is the unique root of $J(\alpha,g_{\omega,\beta
,\gamma}(r))=\rho,$ situated in $(0,\upsilon^{\prime}_{\omega,\beta,\gamma,1}).$ Finally, using again the interlacing inequalities (\ref{ucv32}), we obtain the inequality
$$\frac{\partial}{\partial \alpha} J(\alpha,g_{\omega,\beta,\gamma}(r))=\sum_{n \geq 1}\frac{2r^{2}}{\upsilon^{2}_{\omega,\beta,\gamma,n}-r^{2}}-\sum_{n \geq 1}\frac{2r^{2}}{\upsilon^{\prime 2}_{\omega,\beta,\gamma,n}-r^{2}}<0,$$
where $\gamma>0, \alpha \geq 0$ and $r \in (0, \upsilon^{\prime}_{\omega,\beta,\gamma,1}).$ This implies that the function $\alpha \rightarrow J(\alpha,g_{\omega,\beta,\gamma}(r))$ is strictly decreasing on $[0,\infty)$ for all $\gamma>0, \alpha \geq 0$ and $r \in (0, \upsilon^{\prime}_{\omega,\beta,\gamma,1})$ fixed. Consequently, as a function of $\alpha$ the unique root of the equation $J(\alpha,g_{\omega,\beta,\gamma}(r))=\rho$ is strictly decreasing where $\rho \in [0,1),\gamma>0$ and $r \in (0, \upsilon^{\prime}_{\omega,\beta,\gamma,1}).$ Thus in the case when $\alpha \in (0,1)$ the radius of $\alpha-$convexity of the function $g_{\omega,\beta,\gamma}$ will be between the radius of convexity and the radius of starlikeness of the function $g_{\omega,\beta,\gamma}.$ 
\end{proof}

\begin{theorem}\label{theorem3}
If $(\frac{1}{\omega},\beta) \in ,\gamma>0,  \alpha \geq 0,\rho
\in [0,1),$
 then the radius of $\alpha-$convexity of order $\rho$ of the function $h_{\omega,\beta,\gamma}$  is the smallest positive root of the equation
 $$(1-\alpha)\frac{rh^{\prime}_{\omega,\beta,\gamma}(r)}{h_{\omega,\beta,\gamma}(r)}+\alpha\left(1+\frac{r h^{\prime\prime}_{\omega,\beta,\gamma}(r)}{h^{\prime}_{\omega,\beta,\gamma}(r)}\right)=\rho.$$  The radius of $\alpha-$convexity satisfies $r_{\alpha,\rho}(h_{\omega,\beta,\gamma})<\upsilon^{\prime}_{\omega,\beta,\gamma,1}<\upsilon_{\omega,\beta,\gamma,1},$ where $\upsilon_{\omega,\beta,\gamma,1}$ and $\upsilon^{\prime}_{\omega,\beta,\gamma,1}$ denote the first positive zeros of $h_{\omega,\beta,\gamma}$ and $h^{\prime}_{\omega,\beta,\gamma}$ respectively. Moreover the function $\alpha \rightarrow r_{\alpha,\rho}(h_{\omega,\beta,\gamma})$ is strictly decreasing on $[0,\infty)$ and  consequently, we have $r^{c}_{\rho}<r_{\alpha,\rho}(h_{\omega,\beta,\gamma})<r^{*}_{\rho}(h_{\omega,\beta,\gamma})$ for all $\alpha \in (0,1),\rho \in [0,1)$.
\end{theorem}
\begin{proof}
From the infinite product representation \cite{prajapati}
$$1+\frac{zh^{\prime\prime}_{\omega,\beta,\gamma}(z)}{h^{\prime}_{\omega,\beta,\gamma}(z)}=1-\sum_{n \geq 1}\frac{z}{\upsilon_{\omega,\beta,\gamma,n}-z},$$
where $\upsilon_{\omega,\beta,\gamma,n}$ is the $n$th positive zero of the $h_{\omega,\beta,\gamma}.$ Thus, we have
\begin{eqnarray*}
J(\alpha,h_{\omega,\beta,\gamma}(z))&=&(1-\alpha)\frac{zh^{\prime}_{\omega,\beta,\gamma}(z)}{h_{\omega,\beta,\gamma}(z)}+\alpha \left(1+\frac{zh^{\prime\prime}_{\omega,\beta,\gamma}(z)}{h^{\prime}_{\omega,\beta,\gamma}(z)}\right)\\ \nonumber &=& 1 +\left(\alpha-1 \right)\sum_{n \geq 1}\frac{z}{\upsilon_{\omega,\beta,\gamma,n}-z}-\alpha \sum_{n \geq 1}\frac{z}{\upsilon^{\prime }_{\omega,\beta,\gamma,n}-z}.
\end{eqnarray*}
Applying the inequality (\ref{ucv21}), we have
\begin{eqnarray*}
\frac{1}{\alpha}\Re(J(\alpha,h_{\omega,\beta,\gamma}(z)))&\geq & \frac{1}{\alpha}+\left(1-\frac{1}{\alpha}\right)\sum_{n \geq 1}\frac{r}{\upsilon_{\omega,\beta,\gamma,n}-r}- \sum_{n \geq 1}\frac{r}{\upsilon^{\prime }_{\omega,\beta,\gamma,n}-r}\\\nonumber &=& \frac{1}{\alpha}(J(\alpha,h_{\omega,\beta,\gamma}(r))),
\end{eqnarray*}
where $|z|=r.$
The above inequality implies that for $r \in (0,\upsilon^{\prime}_{\omega,\beta,\gamma,1}),$ we have $$\inf_{z \in \mathbb{D}(r)}J(\alpha,h_{\omega
,\beta,\gamma}(z))=J(\alpha,h_{\omega,\beta,\gamma}(r))$$ and the function $r \rightarrow J(\alpha,h_{\omega,\beta,\gamma}(r))$ is strictly decreasing on $(0,\upsilon^{\prime}_{\omega,\beta,\gamma,1}).$ Since
\begin{eqnarray*}
\frac{\partial }{\partial  r}J(\alpha,h_{\omega,\beta,\gamma}(r))&=&\left(\alpha-1 \right)\sum_{n\geq 1}\frac{\upsilon_{\omega,\beta,\gamma,n}}{(\upsilon_{\omega,\beta,\gamma,n}-r)^{2}}-\alpha\sum_{n\geq 1}\frac{\upsilon^{\prime}_{\omega,\beta,\gamma,n}}{(\upsilon^{\prime}_{\omega,\beta,\gamma,n}-r)^{2}}<0,
\end{eqnarray*}
for $\gamma>0$ and $r \in (0,\upsilon^{\prime}_{\omega,\beta,\gamma,1}).$ Again we  used that the zeros $\upsilon_{\omega,\beta,\gamma,n}$ and $\upsilon^{\prime}_{\omega,\beta,\gamma,n}$ are interlaced and for  all $n \in \mathbb{N},\gamma>0,$ and $r<(\upsilon_{\omega,\beta,\gamma,1}.~\upsilon^{\prime}_{\omega,\beta,\gamma,1})$ we have that
$$\upsilon_{\omega,\beta,\gamma,n}(\upsilon^{\prime}_{\omega,\beta,\gamma,n}-r)^{2}<\upsilon^{\prime}_{\omega,\beta,\gamma,n}(\upsilon_{\omega,\beta,\gamma,n}-r)^{2}.$$
We also have that $ \lim_{r\searrow 0}J(\alpha,h_{\omega,\beta,\gamma}(r))=1>\rho $ and $\lim_{r\nearrow \upsilon^{\prime}_{\omega,\beta,\gamma,1}}J(\alpha,h_{\omega,\beta,\gamma}(r))=-\infty,$
which means that for $z \in \mathbb{D}(r_{3}).$ We have $\Re J(\alpha,h_{\omega,\beta
,\gamma}(r))>\rho$ if and only if $r_{3}$ is the unique root of $J(\alpha,h_{\omega,\beta
,\gamma}(r))=\rho,$ situated in $(0,\upsilon^{\prime}_{\omega,\beta,\gamma,1}).$ Finally using again the interlacing inequalities (\ref{ucv32}), we obtain the inequality
$$\frac{\partial}{\partial \alpha} J(\alpha,h_{\omega,\beta,\gamma}(r))=\sum_{n \geq 1}\frac{r}{\upsilon_{\omega,\beta,\gamma,n}-r}-\sum_{n \geq 1}\frac{r}{\upsilon^{\prime}_{\omega,\beta,\gamma,n}-r}<0,$$
where $\gamma>0, \alpha \geq 0$ and $r \in (0, \upsilon^{\prime}_{\omega,\beta,\gamma,1}).$ This implies that the function $\alpha \rightarrow J(\alpha,h_{\omega,\beta,\gamma}(r))$ is strictly decreasing on $[0,\infty)$ for all $\gamma>0, \alpha \geq 0$ and $r \in (0, \upsilon^{\prime}_{\omega,\beta,\gamma,1})$ fixed. Consequently, as a function of $\alpha$ the unique root of the equation $J(\alpha,h_{\omega,\beta,\gamma}(r))=\rho$ is strictly decreasing where $\rho \in [0,1),\gamma>0$ and $r \in (0, \upsilon^{\prime}_{\omega,\beta,\gamma,1}).$ Thus in the case when $\alpha \in (0,1)$ the radius of $\alpha-$convexity of the function $h_{\omega,\beta,\gamma}$ will be between the radius of convexity and the radius of starlikeness of the function $h_{\omega,\beta,\gamma}.$ 
\end{proof}
\subsection{Radii of $\eta-$parabolic starlikeness of order $\rho$ for the functions $f_{\omega,\beta,\gamma},g_{\omega,\beta,\gamma},$ and $h_{\omega,\beta,\gamma}$}

Now, our aim is to investigate the radii of $\eta-$parabolic starlikeness of order $\rho$ of the normalized forms of the generalized three parameter Mittag-Leffler function, that is of $f_{\omega,\beta,\gamma},$ $g_{\omega,\beta,\gamma}$ and $h_{\omega,\beta,\gamma}$  which are actually solutions of some transcendental equations. For simplicity we use the notation $\lambda(\omega,\beta,\gamma,z)=\phi(\omega,\beta,\gamma,-z^2)$ for this theorem. The technique of determining the radii of $\eta-$parabolic starlikeness of order $\rho$ in the next theorem follows the ideas comes from \cite{bohra}. The results of the next theorem are natural extensions of some recent results see \cite{prajapat}, where the special case of $\gamma=1,\eta=0$ was considered.

\begin{theorem}\label{theorem111}
Let $\left(\frac{1}{\omega},\beta\right)\in W_i,$ $\gamma>0$ and $\rho \in [0,1),\eta \geq 0.$
\begin{enumerate}
\item[\bf a.] The radius of $\eta-$parabolic starlikeness of order $\rho$ of $f_{\omega,\beta,\gamma}$ is $r^{*}_{\rho}(f_{\omega,\beta,\gamma})=x_{\omega,\beta,\gamma,1},$ where $x_{\omega,\beta,\gamma,1}$ is the smallest positive zero of the transcendental equation
$$(1+\eta)r\lambda'(\omega,\beta,\gamma,r)-\beta (\rho-1)\lambda(\omega,\beta,\gamma,r)=0.$$
\item[\bf b.] The radius of $\eta-$parabolic starlikeness of order $\rho$ of $g_{\omega,\beta,\gamma}$ is $r^{*}_{\rho}(g_{\omega,\beta,\gamma})=y_{\omega,\beta,\gamma,1},$ where $y_{\omega,\beta,\gamma,1}$ is the smallest positive zero of the transcendental equation
$$(1+\eta)r\lambda'(\omega,\beta,\gamma,r)- (\rho-1)\lambda(\omega,\beta,\gamma,r)=0.$$
\item[\bf c.] The radius of $\eta-$parabolic starlikeness of order $\rho$ of $h_{\omega,\beta,\gamma}$ is $r^{*}_{\rho}(h_{\omega,\beta,\gamma})=z_{\omega,\beta,\gamma,1},$ where $z_{\omega,\beta,\gamma,1}$ is the smallest positive zero of the transcendental equation
$$(1+\eta)\sqrt{r}\lambda'(\omega,\beta,\gamma,\sqrt{r})-2(\rho-1)\lambda(\omega,\beta,\gamma,\sqrt{r})=0.$$
\end{enumerate}
\end{theorem}
\begin{proof}
Recall that  the zeros of the Mittag-Leffler function $\phi(\omega,\beta,\gamma,z)$ are all real and infinite product exists. Now from the infinite product representation was proved in \cite{prajapati} which of the form
$$\phi(\omega,\beta,\gamma,-z^{2})=\frac{1}{\Gamma(\beta)} \prod_{n\geq 1}\left(1-\frac{z^{2}}{\lambda^{2}_{\omega,\beta,\gamma,n}}\right)$$
and this infinite product is uniformly convergent on each compact subset of $\mathbb{C}.$ Denoting, the above expression by $\lambda(\omega,\beta,\gamma,z),$ and by logarithmic differentiation we get
\begin{equation*}
\frac{\lambda'(\omega,\beta,\gamma,z)}{\lambda(\omega,\beta,\gamma,z)}=\sum _{n \geq 1}\frac{-2z}{\lambda^{2}_{\omega,\beta,\gamma,n}-z^{2}},
\end{equation*}
which in turn implies that
\begin{equation}\label{para1}
\frac{zf^{\prime}_{\omega,\beta,\gamma}(z)}{f_{\omega,\beta,\gamma}(z)}=1-\frac{1}{\beta}\sum _{n \geq 1}\frac{2z^{2}}{\lambda^{2}_{\omega,\beta,\gamma,n}-z^{2}}. 
\end{equation}
\begin{equation}\label{para2}
\frac{zg^{\prime}_{\omega,\beta,\gamma}(z)}{g_{\omega,\beta,\gamma}(z)}=1-\sum _{n \geq 1}\frac{2z^{2}}{\lambda^{2}_{\omega,\beta,\gamma,n}-z^{2}}. 
\end{equation}
\begin{equation}\label{para3}
\frac{zh^{\prime}_{\omega,\beta,\gamma}(z)}{h_{\omega,\beta,\gamma}(z)}=1-\sum _{n \geq 1}\frac{z}{\lambda^{2}_{\omega,\beta,\gamma,n}-z}.
\end{equation}
We know that \cite{baricz2} if $z \in \mathbb{C}$ and $\theta \in \mathbb{R}$ are such that $\theta > |z|,$ then
\begin{equation}\label{cov}
\frac{|z|}{\theta-|z|} \geq \Re \left(\frac{z}{\theta-z}\right).
\end{equation}
Thus the inequality
\begin{equation*}
\frac{|z|^{2}}{\lambda^{2}_{\omega,\beta,\gamma,n}-|z|^{2}} \geq \Re \left(\frac{z^{2}}{\lambda^{2}_{\omega,\beta,\gamma,n}-z^{2}}\right),
\end{equation*}
is valid for every $\left(\frac{1}{\omega},\beta\right)\in W_i,$ $\gamma>0,$ $n \in \mathbb{N}$ and $|z|< \lambda_{\omega,\beta,\gamma,1},$ and therefore
under the same conditions we have that
$$
\Re \left(\frac{zf^{\prime}_{\omega,\beta,\gamma}(z)}{f_{\omega,\beta,\gamma}(z)}\right)=1-\frac{1}{\beta}\Re \left(\sum _{n \geq 1}\frac{2z^{2}}{\lambda^{2}_{\omega,\beta,\gamma,n}-z^{2}}\right)\geq 1-\frac{1}{\beta}\sum _{n \geq 1}\frac{2|z|^{2}}{\lambda^{2}_{\omega,\beta,\gamma,n}-|z|^{2}}=\frac{|z|f^{\prime}_{\omega,\beta,\gamma}(|z|)}{f_{\omega,\beta,\gamma}(|z|)},
$$
$$
\Re \left(\frac{zg^{\prime}_{\omega,\beta,\gamma}(z)}{g_{\omega,\beta,\gamma}(z)}\right)=1-\Re \left(\sum _{n \geq 1}\frac{2z^{2}}{\lambda^{2}_{\omega,\beta,\gamma,n}-z^{2}}\right)\geq 1-\sum _{n \geq 1}\frac{2|z|^{2}}{\lambda^{2}_{\omega,\beta,\gamma,n}-|z|^{2}}=\frac{|z|g^{\prime}_{\omega,\beta,\gamma}(|z|)}{g_{\omega,\beta,\gamma}(|z|)},$$
and
$$
\Re \left(\frac{zh^{\prime}_{\omega,\beta,\gamma}(z)}{h_{\omega,\beta,\gamma}(z)}\right)=1-\Re \left(\sum _{n \geq 1}\frac{z}{\lambda^{2}_{\omega,\beta,\gamma,n}-z}\right)\geq 1-\sum _{n \geq 1}\frac{|z|}{\lambda^{2}_{\omega,\beta,\gamma,n}-|z|}=\frac{|z|h^{\prime}_{\omega,\beta,\gamma}(|z|)}{h_{\omega,\beta,\gamma}(|z|)}.
$$
Now using triangle inequality $\left||z_{1}|-|z_{2}|\right| \leq \left|z_{1}-z_{2}\right|$ we get 
$$\left|\frac{zf^{\prime}_{\omega,\beta,\gamma}(z)}{f_{\omega,\beta,\gamma}(z)}-1\right|=\frac{1}{\beta}\left|\sum_{n \geq 1} \frac{2z^{2}}{\lambda^{2}_{\omega,\beta,\gamma,n}-z^{2}}\right| \leq \frac{1}{\beta}\sum_{n \geq 1}\frac{2|z|^{2}}{\lambda^{2}_{\omega,\beta,\gamma,n}-|z|^{2}}=1-\frac{|z|f^{\prime}_{\omega,\beta,\gamma}(|z|)}{f_{\omega,\beta,\gamma}(|z|)},$$
 $$\left|\frac{zg^{\prime}_{\omega,\beta,\gamma}(z)}{g_{\omega,\beta,\gamma}(z)}-1\right|=\left|\sum_{n \geq 1} \frac{2z^{2}}{\lambda^{2}_{\omega,\beta,\gamma,n}-z^{2}}\right| \leq \sum_{n \geq 1}\frac{2|z|^{2}}{\lambda^{2}_{\omega,\beta,\gamma,n}-|z|^{2}}=1-\frac{|z|g^{\prime}_{\omega,\beta,\gamma}(|z|)}{g_{\omega,\beta,\gamma}(|z|)},$$
 and
 $$\left|\frac{zh^{\prime}_{\omega,\beta,\gamma}(z)}{h_{\omega,\beta,\gamma}(z)}-1\right|=\left|\sum_{n \geq 1} \frac{z}{\lambda^{2}_{\omega,\beta,\gamma,n}-z}\right| \leq \sum_{n \geq 1}\frac{|z|}{\lambda^{2}_{\omega,\beta,\gamma,n}-|z|}=1-\frac{|z|f^{\prime}_{\omega,\beta,\gamma}(|z|)}{f_{\omega,\beta,\gamma}(|z|)}.$$
 Hence we have
 $$\Re \left(\frac{zf^{\prime}_{\omega,\beta,\gamma}(z)}{f_{\omega,\beta,\gamma}(z)}\right)-\eta \left|\frac{zf^{\prime}_{\omega,\beta,\gamma}(z)}{f_{\omega,\beta,\gamma}(z)}-1\right|-\rho \geq(1+\eta) \frac{|z|f^{\prime}_{\omega,\beta,\gamma}(|z|)}{f_{\omega,\beta,\gamma}(|z|)}-(\eta+\rho), $$
 $$\Re \left(\frac{zg^{\prime}_{\omega,\beta,\gamma}(z)}{g_{\omega,\beta,\gamma}(z)}\right)-\eta \left|\frac{zg^{\prime}_{\omega,\beta,\gamma}(z)}{g_{\omega,\beta,\gamma}(z)}-1\right|-\rho \geq(1+\eta) \frac{|z|g^{\prime}_{\omega,\beta,\gamma}(|z|)}{g_{\omega,\beta,\gamma}(|z|)}-(\eta+\rho),$$
 and
 $$\Re \left(\frac{zh^{\prime}_{\omega,\beta,\gamma}(z)}{h_{\omega,\beta,\gamma}(z)}\right)-\eta \left|\frac{zh^{\prime}_{\omega,\beta,\gamma}(z)}{h_{\omega,\beta,\gamma}(z)}-1\right|-\rho \geq(1+\eta) \frac{|z|h^{\prime}_{\omega,\beta,\gamma}(|z|)}{h_{\omega,\beta,\gamma}(|z|)}-(\eta+\rho),$$
where equalities are attained only when $z=|z|=r.$ The minimum principle for harmonic functions and the previous inequalities imply that the corresponding inequalities in the above are valid if and only if we have $|z|<x_{\alpha,\beta,\gamma,1},$ $|z|< y_{\alpha,\beta,\gamma,1}$ and $|z|< z_{\alpha,\beta,\gamma,1},$ respectively, where $x_{\alpha,\beta,\gamma,1},$ $y_{\alpha,\beta.\gamma,1}$ and $z_{\alpha,\beta,\gamma,1}$ are the smallest positive roots of the following equations
$$
(1+\eta) \frac{rf^{\prime}_{\omega,\beta,\gamma}(r)}{f_{\omega,\beta,\gamma}(r)}-(\eta+\rho)=0,$$
$$(1+\eta) \frac{rg^{\prime}_{\omega,\beta,\gamma}(r)}{g_{\omega,\beta,\gamma}(r)}-(\eta+\rho)=0,$$
$$(1+\eta) \frac{rh^{\prime}_{\omega,\beta,\gamma}(r)}{h_{\omega,\beta,\gamma}(r)}-(\eta+\rho)=0,$$
which are equivalent to $$(1+\eta)r\lambda'(\omega,\beta,\gamma,r)-\beta(\rho-1)\lambda(\omega,\beta,\gamma,r)=0,$$ 
$$(1+\eta) r\lambda'(\omega,\beta,\gamma,r)- (\rho-1)\lambda(\omega,\beta,\gamma,r)=0,$$
$$(1+\eta)\sqrt{r}\lambda'(\omega,\beta,\gamma,\sqrt{r})-2(\rho-1)\lambda(\omega,\beta,\gamma,\sqrt{r})=0.$$
We note that $$\lim_{r \searrow 0}  (1+\eta)\left(1-\frac{1}{\beta}\sum_{n \geq 1}\frac{2r^{2}}{\lambda^{2}_{\omega,\beta,\gamma,n}-r^{2}}\right)-(\eta+\rho)=1-\rho>0$$
and 
$$\lim_{r \nearrow \lambda_{\omega,\beta,\gamma,1}}  (1+\eta)\left(1-\frac{1}{\beta}\sum_{n \geq 1}\frac{2r^{2}}{\lambda^{2}_{\omega,\beta,\gamma,n}-r^{2}}\right)-(\eta+\rho)=-\infty.$$
Hence we have $(1+\eta)r\lambda'(\omega,\beta,\gamma,r)-\beta(\rho-1)\lambda(\omega,\beta,\gamma,r)=0$ has a root in $(0,\lambda_{\omega,\beta,\gamma,1}).$ Similarly for other two equations.
\end{proof}
\begin{remark}
For $\eta=0,$ Theorem \ref{theorem111} reduces to \cite[Theorem-1]{prajapati}.
\end{remark}
\subsection{Radii of strong starlikeness of order $\rho$ for the functions $f_{\omega,\beta,\gamma},g_{\omega,\beta,\gamma},$ and $h_{\omega,\beta,\gamma}$}
Our aim is to investigate the radii of strong  starlikeness of order $\rho$ of the normalized forms of the generalized three parameter Mittag-Leffler function, that is of $f_{\omega,\beta,\gamma},$ $g_{\omega,\beta,\gamma}$ and $h_{\omega,\beta,\gamma}$  which are the solutions of some  equations.  The technique of determining the radii of strong starlikeness of order $\rho$ in the next theorem follows the ideas comes from \cite{bohra}. 
\begin{definition}\cite{brannan} A function  $f(z)\in \mathcal{S}$ is said to  the class of univalent  \textit{strongly starlike of order $\rho$} in $\mathbb{D}$ if it satisfies the inequalities:$$\left|\arg \left(\frac{zf^{\prime}(z)}{f(z)}\right)\right|<\frac{\pi \rho}{2}~~~~~(0<\rho \leq 1, z \in \mathbb{D}).$$
\end{definition}
The radius of strong starlikeness of order $\rho$ of the function $f$ is defined in \cite{gang}.
$$r^{S^{*}}_{\rho}(f)=\left\{r \in (0,r_{f}):\left|\arg \left(\frac{zf^{\prime}(z)}{f(z)}\right)\right|<\frac{\pi \rho}{2},0<\rho \leq 1, z \in \mathbb{D}(r)\right\}.$$
\begin{lemma}\label{strong}\cite{gang}
If $c$ is any point in $|\arg w| \leq \frac{\pi}{2}\rho$ and if $R_{c}\leq \Re[c]\sin\frac{\pi}{2}\rho-\Im[c]\cos\frac{\pi}{2}\rho,~~ \Im[c]\geq 0.$ The disk $|w-c|\leq R_{c}$ is contained in the sector $|\arg w|\leq \frac{\pi}{2}\rho,~~ 0<\rho \leq 1.$ In particular when $\Im[c]=0,$ the condition becomes $R_{c}\leq c \sin\frac{\pi}{2}\rho.$
\end{lemma}
The above definition and lemma are used to determine the radii of strong  starlikeness of order $\rho$ for the functions of the form (\ref{uncv}).
\begin{theorem}\label{strong1}
Let $\left(\frac{1}{\omega},\beta\right)\in W_i,$ $\gamma>0, 0 < \rho \leq 1.$ The following are true.
\begin{enumerate}
\item[\bf a.] The radius of strong starlikeness of order $\rho$ of the $f_{\omega,\beta,\gamma}$ is  the smallest positive zero of $\psi(r)=0$ in $(0,\lambda_{\omega,\beta,\gamma,1})$ where
$$\psi(r)=\frac{2}{\beta}\sum_{n \geq 1}\frac{r^{2}(\lambda^{2}_{\omega,\beta,\gamma,n}+r^{2}\sin \frac{\pi}{2}\rho)}{\lambda^{4}_{\omega,\beta,\gamma,n}-r^{4}}-\sin \frac{\pi}{2}\rho.$$
\item[\bf b.] The radius of strong  starlikeness of order $\rho$ of the $g_{\omega,\beta,\gamma}$  is the smallest positive zero of $\phi(r)=0$ in $(0,\lambda_{\omega,\beta,\gamma,1})$ where
 $$\phi(r)=2\sum_{n \geq 1}\frac{r^{2}(\lambda^{2}_{\omega,\beta,\gamma,n}+r^{2}\sin \frac{\pi}{2}\rho)}{\lambda^{4}_{\omega,\beta,\gamma,n}-r^{4}}-\sin \frac{\pi}{2}\rho.$$
\item[\bf c.] The radius of strong  starlikeness of  order $\rho$ of the $h_{\omega,\beta,\gamma}$ is  the smallest positive zero $\varphi(r)=0$ in $(0,\lambda_{\omega,\beta,\gamma,1})$ where
 $$\varphi(r)=\sum_{n \geq 1}\frac{r(\lambda^{2}_{\omega,\beta,\gamma,n}+r\sin \frac{\pi}{2}\rho)}{\lambda^{4}_{\omega,\beta,\gamma,n}-r^{2}}-\sin \frac{\pi}{2}\rho.$$
\end{enumerate}
\end{theorem}
\begin{proof}
For $|z|\leq r <,|z_{k}|<R>r,$ we have from \cite[Lemma-3.2]{gang}
\begin{equation}\label{strong2}
\left|\frac{z}{z-z_{k}}+\frac{r^{2}}{R^{2}-r^{2}}\right|\leq \frac{Rr}{R^{2}-r^{2}}.
\end{equation}
Since the series $\sum_{n\geq 1}\frac{2r^{2}}{\lambda^{2}_{\omega,\beta,\gamma,n}-r^{2}}$ and $\sum_{n \geq 1}\frac{r}{\lambda^{2}_{\omega,\beta,\gamma,n}-r}$ are convergent, from (\ref{strong2}), (\ref{para1}), (\ref{para2}), and (\ref{para3}) we have
\begin{equation}\label{strong3}
\left|\frac{zf^{\prime}_{\omega,\beta,\gamma}(z)}{f_{\omega,\beta,\gamma}(z)}-\left(1-\frac{2}{\beta}\sum_{n\geq 1}\frac{r^{4}}{\lambda^{4}_{\omega,\beta,\gamma,n}-r^{4}}\right)\right| \leq  \frac{2}{\beta}\sum_{n\geq 1}\frac{\lambda^{2}_{\omega,\beta,\gamma,n} r^{2}}{\lambda^{4}_{\omega,\beta,\gamma,n}-r^{4}}
\end{equation}
\begin{equation}\label{strong4}
\left|\frac{zg^{\prime}_{\omega,\beta,\gamma}(z)}{g_{\omega,\beta,\gamma}(z)}-\left(1-2\sum_{n\geq 1}\frac{r^{4}}{\lambda^{4}_{\omega,\beta,\gamma,n}-r^{4}}\right)\right| \leq  2\sum_{n\geq 1}\frac{\lambda^{2}_{\omega,\beta,\gamma,n} r^{2}}{\lambda^{4}_{\omega,\beta,\gamma,n}-r^{4}}
\end{equation}
\begin{equation}\label{strong5}
\left|\frac{zh^{\prime}_{\omega,\beta,\gamma}(z)}{h_{\omega,\beta,\gamma}(z)}-\left(1-\sum_{n\geq 1}\frac{r^{2}}{\lambda^{4}_{\omega,\beta,\gamma,n}-r^{2}}\right)\right| \leq  \sum_{n\geq 1}\frac{\lambda^{2}_{\omega,\beta,\gamma,n} r}{\lambda^{4}_{\omega,\beta,\gamma,n}-r^{2}}
\end{equation}
where equalities are attained only when $z=|z|=r.$ For $z \in \mathbb{D}(\lambda_{\omega,\beta,\gamma,1})$ and $\lambda_{\omega,\beta,\gamma,n}$ denotes the $nth$ positive zero of the Mittag-Leffler function. From the Lemma \ref{strong}, we see that the disc given by (\ref{strong3}) is contained in the sector $|\arg w|\leq \frac{\pi}{2}\rho,$ if
$$\frac{2}{\beta}\sum_{n\geq 1}\frac{\lambda^{2}_{\omega,\beta,\gamma,n} r^{2}}{\lambda^{4}_{\omega,\beta,\gamma,n}-r^{4}} \leq \left(1-\frac{2}{\beta}\sum_{n\geq 1}\frac{r^{4}}{\lambda^{4}_{\omega,\beta,\gamma,n}-r^{4}}\right)\sin \frac{\pi}{2}\rho$$
is satisfied. The above inequalities simplifies to $\psi(r)\leq 0$ where
  $$\psi(r)=\frac{2}{\beta}\sum_{n \geq 1}\frac{r^{2}(\lambda^{2}_{\omega,\beta,\gamma,n}+r^{2}\sin \frac{\pi}{2}\rho)}{\lambda^{4}_{\omega,\beta,\gamma,n}-r^{4}}-\sin \frac{\pi}{2}\rho.$$
  We note that
  $$\psi^{\prime}(r)=\frac{2}{\beta}\sum_{n \geq 1}\frac{2r\lambda^{6}_{\omega,\beta,\gamma,n}+2r^{5}\lambda^{2}_{\omega,\beta,\gamma,n}+4r^{3}\lambda^{4}_{\omega,\beta,\gamma,n}\sin \frac{\pi}{2}\rho}{(\lambda^{4}_{\omega,\beta,\gamma,n}-r^{4})^{2}} \geq 0.$$
  Also $\lim_{r\searrow 0}\psi(r)<0$ and $\lim _{r \nearrow \lambda_{\omega,\beta,\gamma,1}}\psi(r)=\infty.$ Thus $\psi(r)=0$ has a unique root $R_{f}$ in $(0,\lambda_{\omega,\beta,\gamma,1}).$ Hence the function $f_{\omega,\beta,\gamma}$ is strongly starlike in $|z|<R_{f}.$ This completes the proof of part (a).

  The disc given by (\ref{strong4}) is contained in the sector $|\arg w|\leq \frac{\pi}{2}\rho,$ if
  $$\phi(r)=2\sum_{n \geq 1}\frac{r^{2}(\lambda^{2}_{\omega,\beta,\gamma,n}+r^{2}\sin \frac{\pi}{2}\rho)}{\lambda^{4}_{\omega,\beta,\gamma,n}-r^{4}}-\sin \frac{\pi}{2}\rho.$$
   Also $\lim_{r\searrow 0}\phi(r)<0$ and $\lim _{r \nearrow \lambda_{\omega,\beta,\gamma,1}}\phi(r)=\infty.$ This completes the proof of part(b).

   The disc given by (\ref{strong5}) is contained in the sector $|\arg w|\leq \frac{\pi}{2}\rho,$ if
  $$\varphi(r)=\sum_{n \geq 1}\frac{r(\lambda^{2}_{\omega,\beta,\gamma,n}+r\sin \frac{\pi}{2}\rho)}{\lambda^{4}_{\omega,\beta,\gamma,n}-r^{2}}-\sin \frac{\pi}{2}\rho.$$
   Also $\lim_{r\searrow 0}\varphi(r)<0$ and $\lim _{r \nearrow \lambda_{\omega,\beta,\gamma,1}}\varphi(r)=\infty.$ This completes the proof of part(c). 
\end{proof}

\end{document}